%% file: dr-arxiv.tex
\documentclass[11pt,a4paper,reqno]{article}
\usepackage{mystyle}
\usepackage{fullpage}

\newcommand{\PG}{\proj_{T_{\vsol}^{G}}}
\newcommand{\PGs}{\proj_{S_{\vsol}^{G}}}
\newcommand{\PJ}{\proj_{T_{\xsol}^{J}}}
\newcommand{\PJs}{\proj_{S_{\xsol}^{J}}}

\newcommand{\BG}{W_G}
\newcommand{\CJ}{W_J}
\newcommand{\dder}{\mathrm{D}}

\newcommand{\bPG}{\proj_{\bmT_{\bmx^\star}^{\bmG}}}

\newcommand{\bPJ}{\proj_{\bmT_{\bmx^\star}^{\bmJ}}}

\graphicspath{{./figure/}}

\begin{document}

\title{Activity Identification and Local Linear Convergence of Douglas--Rachford/ADMM under Partial Smoothness}
\author{
	Jingwei Liang\thanks{Jingwei Liang, Jalal Fadili: GREYC, CNRS-ENSICAEN-Universit\'{e} de Caen, 
	E-mail: \{Jingwei.Liang, Jalal.Fadili\}@ensicaen.fr}
	, Jalal Fadili\samethanks 
	, Gabriel Peyr\'{e}\thanks{Gabriel Peyr\'{e}: CNRS, CEREMADE, Universit\'{e} Paris-Dauphine, 
	E-mail: Gabriel.Peyre@ceremade.dauphine.fr}
	, Russell Luke\thanks{Russell Luke: Institut f\"{u}r Numerische und Angewandte Mathematik Universit\"{a}t G\"{o}ttingen, 
	E-mail: r.luke@math.uni-goettingen.de}
}
\date{}
\maketitle

\begin{abstract}
Convex optimization has become ubiquitous in most quantitative disciplines of science, including variational image processing. Proximal splitting algorithms are becoming popular to solve such structured convex optimization problems. Within this class of algorithms, Douglas--Rachford (DR) and alternating direction method of multipliers (ADMM) are designed to minimize the sum of two proper lower semi-continuous convex functions whose proximity operators are easy to compute. The goal of this work is to understand the local convergence behaviour of DR (resp. ADMM) when the involved functions (resp. their Legendre-Fenchel conjugates) are moreover partly smooth. More precisely, when both of the two functions (resp. their conjugates) are partly smooth relative to their respective manifolds, we show that DR (resp. ADMM) identifies these manifolds in finite time. Moreover, when these manifolds are affine or linear, we prove that DR/ADMM is locally linearly convergent. When $J$ and $G$ are locally polyhedral, we show that the optimal convergence radius is given in terms of the cosine of the Friedrichs angle between the tangent spaces of the identified manifolds. This is illustrated by several concrete examples and supported by numerical experiments.
\end{abstract}



\input{tex/sec-introduction}

\input{tex/sec-partial-smooth}

\input{tex/sec-identification}

\input{tex/sec-linear-rate}

\input{tex/sec-m-functions}

\input{tex/sec-experiment}

\input{tex/sec-conclusion}

\section*{Acknowledgments}
This work has been partly supported by the European Research Council (ERC project SIGMA-Vision). JF was partly supported by Institut Universitaire de France.


\bibliographystyle{plain}
\bibliography{dr-arxiv}

\end{document}

%% file: tex/sec-introduction.tex
\section{Introduction}
\label{sec:intro}

\subsection{Problem formulation}
In this work, we consider the problem of solving
\begin{equation}\label{eq:minP}
\min_{x \in \RR^n} J(x) + G(x)  ,
\end{equation}
where both $J$ and $G$ are in $\Gamma_0(\RR^n)$, the class of proper, lower semi-continuous (lsc) and convex functions. We assume that $\ri\Pa{\dom(J)} \cap \ri\Pa{\dom(G)} \neq \emptyset$, where $\ri(C)$ is the relative interior of the nonempty convex set $C$, and $\dom(F)$ is the domain of the function $F$. We also assume that the set of minimizers is non-empty, and that these two functions are simple, meaning that their respective proximity operators, $\prox_{\gamma J}$ and $\prox_{\gamma G}$, $\gamma > 0$, are easy to compute, either exactly or up to a very good approximation. Problem~\eqref{eq:minP} covers a large number of problems including those appearing in variational image processing (see Section \ref{sec:DRapp}).

An efficient and provably convergent algorithm to solve this class of problems is the Douglas--Rachford splitting method \cite{lions1979splitting}, which reads, in its relaxed form,
\begin{equation}\label{eq:drsgen}
\left\{
\begin{aligned}
\vkp &= \prox_{\gamma G}\pa{2\xk - \zk}  , \\
\zkp &= (1-{\lambda_k}) \zk + {\lambda_k}\pa{\zk + \vkp - \xk}  , \\
\xkp &= \prox_{\gamma J} \zkp  ,
\end{aligned}
\right.
\end{equation}
for $\gamma > 0$, $\lambda_k \in ]0,2]$ with $\sum_{k \in \NN}\lambda_k(2-\lambda_k)=+\infty$.
The fixed-point operator $\BDR$ with respect to $\zk$ takes the form
\begin{align*}
\BDR &\eqdef \qfrac12\pa{\rprox_{\gamma G}\circ\rprox_{\gamma J} + \Id}, \\
\rprox_{\gamma J} &\eqdef 2\prox_{\gamma J} - \Id, ~~
\rprox_{\gamma G} \eqdef 2\prox_{\gamma G} - \Id  .
\end{align*}
For a proper lsc convex function $J$, the proximity operator is defined as, for $\gamma>0$,
\[
\prox_{\gamma J}(z) = \argmin_{x\in\RR^n} \gamma J(x)  + \qfrac{1}{2}\norm{x-z}^2 .
\]
Since the set of minimizers of \eqref{eq:minP} is assumed to be non-empty, so is the $\Fix(\BDR)$ since the former is nothing but $\prox_{\gamma J}\Pa{\Fix(\BDR)}$. See \cite{bauschke2011convex} for a more detailed account on DR in real Hilbert spaces.

\begin{rem}
The DR algorithm is not symmetric w.r.t. the order of the functions $J$ and $G$. Nevertheless, the convergence claims above hold true of course when the order of $J$ and $G$ is reversed in \eqref{eq:drsgen}. In turn, all of our statements throughout also extend to this case with minor adaptations. Note also that the standard DR only accounts for the sum of $2$ functions. But extension to more than $2$ functions is straightforward through a product space trick, see Section~\ref{sec:DRmfunc} for details.
\end{rem}

\subsection{Contributions}
Based on the assumption that both $J$ and $G$ are partly smooth relative to smooth manifolds, we show that DR identifies in finite time these manifolds. In plain words, this means that after a finite number of iterations, the iterates $(\xk, \vk)$ lie respectively in the partial smoothness (acyive) manifolds associated to $J$ and $G$ respectively. When these manifolds are affine/linear, we establish local linear convergence of DR. 
Moreover, when both $G$ and $J$ are locally polyhedral, we show that the optimal convergence rate is given in terms of the cosine of the Friedrichs angle between the tangent spaces of the manifolds. We also generalize these claims to the minimization of the sum of more than two functions. We finally exemplify our results with several experiments on variational signal and image processing. 

It is important to note that our results readily apply to ADMM, since it is well-known that ADMM is the DR method applied to the Fenchel dual problem of~\eqref{eq:minP}. More precisely, we only need to assume that the conjugates $J^*$ and $G^*$ are partly smooth. 
Therefore, to avoid unnecessary lengthy repetitions, we only focus in detail on the primal DR splitting method.

\subsection{Relation to prior work}
There are problem instances in the literature where DR was proved to converge locally linearly. For instance, in \cite[Proposition~4]{lions1979splitting}, it was assumed that the "internal" function is strongly convex with a Lipschitz continuous gradient. This local linear convergence result was further investigated in \cite{davis2014convergence,giselsson2014metric} under smoothness and strong convexity assumptions.
On the other hand, for the Basis Pursuit (BP) problem, i.e. $\ell_1$ minimization with an affine constraint, is considered in \cite{DemanetZhang13} and an eventual local linear convergence is shown in the absence of strong convexity. The author in \cite{boley2013local} analyzes the local convergence behaviour of ADMM for quadratic or linear programs, and shows local linear convergence if the optimal solution is unique and the strict complementarity holds. This turns out to be a special case of our framework. For the case of two subspaces, linear convergence of DR with the optimal rate being the cosine of the Friedrichs angle between the subspaces is proved in~\cite{BauschkeDR14}. Our results generalize those of~\cite{DemanetZhang13,boley2013local,BauschkeDR14} to a much larger class of problems.
For the non-convex case, \cite{BorweinSims10} considered DR method for a feasibility problem of a sphere intersecting a line or more generally a proper affine subset. Such feasibility problems with an affine subspace and a super-regular set (in the sense of \cite{LukeMAP09}) with strongly regular intersection was considered in \cite{HesseLuke13}, and was generalized later to two $(\epsilon, \delta)$-regular sets with linearly regular intersection \cite{hesse2013nonconvex}, see also \cite{Phan14} for an even more general setting. However, even in the convex case, the rate provided in \cite{Phan14} is nowhere near the optimal rate given by the Friedrichs angle.
	

\subsection{Notations}
For a nonempty convex set $C \subset \RR^n$, denote $\Aff(C)$ its affine hull, and $\LinHull(C)$ the subspace parallel to $\Aff(C)$. Denote $\proj_C$ the orthogonal projection operator onto $C$ and $N_C$ its normal cone. For $J \in \Gamma_0(\RR^n)$, denote $\partial J$ its subdifferential and $\prox_{\gamma J}$ its proximity operator with $\gamma > 0$. Define the model subspace
\eq{
	T_{x} \eqdef \LinHull\Pa{\partial J(x)}^\perp .
}
It is obvious that $\proj_{T_x}\Pa{\partial J(x)}$ is a singleton, and therefore defined as
\begin{equation}\label{eq:e_x}
e_x \eqdef \proj_{T_x}\Pa{\partial J(x)} .
\end{equation}
%
Suppose $\Mm \subset \RR^n$ is a $C^2$-manifold around $x$, denote $\tgtManif{x}{\Mm}$ the tangent space of $\Mm$ at $x \in \RR^n$.

%% file: tex/sec-partial-smooth.tex
\section{Partly Smooth Functions}
\label{sec:psf}

\subsection{Definition and main properties}
Partial smoothness of functions was originally defined in \cite{Lewis-PartlySmooth}, our definition hereafter specializes it to the case of proper lsc convex functions.

\begin{defn}[Partly smooth function]\label{dfn-partly-smooth}
Let $J \in \Gamma_0(\RR^n)$, and $x\in\RR^n$ such that $\partial J(x) \neq \emptyset$.
$J$ is \emph{partly smooth} at $x$ relative to a set $\Mm$ containing $x$ if 
\begin{enumerate}[label={\rm (\arabic{*})}]
\item\label{PS-C2}\emph{(Smoothness)} 
$\Mm$ is a $C^2$-manifold around $x$, $J|_\Mm$ is $C^2$ near $x$;
\item\label{PS-Sharp}\emph{(Sharpness)} 
The tangent space $\tgtManif{x}{\Mm}$ is $T_{x}$;
\item\label{PS-DiffCont}\emph{(Continuity)} 
The set--valued mapping $\partial J$ is continuous at $x$ relative to $\Mm$.
\end{enumerate}
The class of partly smooth functions at $x$ relative to $\Mm$ is denoted as $\PSF{x}{\Mm}$. When $\Mm$ is either affine or linear, $\Mm = x + T_x$, and we denote this subclass as $\PSFAL{x}{T_x}$.
\end{defn}

Capitalizing on the results of \cite{Lewis-PartlySmooth}, it can be shown that, under mild transversality conditions, the set of lsc convex and partly smooth functions is closed under addition and pre-composition by a linear operator. Moreover, absolutely permutation-invariant convex and partly smooth functions of the singular values of a real matrix, i.e. spectral functions, are convex and partly smooth spectral functions of the matrix \cite{Daniilidis-SpectralIdent}. 

Examples of partly smooth functions that have become very popular recently in the signal processing, optimization, statistics and machine learning literature are $\ell_1$, $\ell_{1,2}$, $\ell_\infty$, total variation (TV) and nuclear norm regularizations. In fact, the nuclear norm is partly smooth at a matrix $x$ relative to the manifold $\Mm = \enscond{x'}{ \rank(x')=\rank(x)}$. The first four regularizers are all part of the class $\PSFAL{x}{T_x}$. 

We now define a subclass of partly smooth functions where the manifold is affine or linear and the vector $e_x$  \eqref{eq:e_x} is locally constant.
\begin{defn}\label{dtf-polyhedral}
$J$ belongs to the class $\PSFLS{x}{T_x}$ if and only if $J \in \PSFAL{x}{T_x}$ and $e_x$ is constant near $x$, i.e. there exists a neighbourhood $\Nn$ of $x$ such that $\forall x' \in \pa{x+T_x} \cap \Nn$,
\[
e_{x'} = e_x .
\]
\end{defn}
The class of functions that conform with this definition is that of locally polyhedral functions \cite[Section~6.5]{vaiter2013model}, which includes for instance the $\ell_1$, $\ell_\infty$ norms and the anisotropic TV semi-norm that are widely used in signal and image processing, computer vision, machine learning and statistics. The indicator function of a polyhedral set is also in $\PSFALS{x}{T_x}$ at each $x$ in the relative interior of one of its faces relative to the affine hull of that face, i.e. $x+T_x=\Aff(\text{Face of $x$})$. Observe that for polyhedral functions, in fact, the subdifferential itself is constant along $x+T_x$.
 
\subsection{Proximity operator}
This part shows that the proximity operator of a partly smooth function can be given in an implicit form. 

\begin{prop}\label{prop:proxps}
Let $p \eqdef \prox_{\gamma \J}(x) \in \Mm$. Assume that $J \in \PSF{p}{\Mm}$. Then for any point $x$ near $p$, we have
\begin{equation*}
p = \proj_{\Mm}(x) - \gamma\e_{p} + o\pa{\norm{x-p}} .
\end{equation*}
In particular, if $J \in \PSFAL{p}{T_p}$, then for any $x \in \RR^n$, we have
\begin{equation*}
p = \proj_{p+T_p}(x) - \gamma\e_{p} .
\end{equation*}
\end{prop}
\begin{proof}
We start with the following lemma whose proof can be found in \cite{LiangFB14}.
\begin{lem}\label{lemma:proj_M}
Suppose that $J \in \PSF{p}{\Mm}$. Then any point $x$ near $p$ has a unique projection $\proj_{\Mm}(x)$, $\proj_{\Mm}$ is $C^1$ around $p$, and thus
\[
\proj_{\Mm}(x) - p = \proj_{T_p}(x-p) + o\pa{\norm{x-p}} .
\]
\end{lem}

Let's now turn to the proof of our proposition. We have the equivalent characterization
\begin{equation}
\label{eq:proxmoninc}
p = \prox_{\gamma \J}(x) \iff x - p \in \gamma \partial \J(p) .
\end{equation}
Projecting \eqref{eq:proxmoninc} on $T_p$ and using Lemma~\ref{lemma:proj_M}, we get
\[
\proj_{T_p}(x-p) = \proj_{\Mm}(x) - p + o\pa{\norm{x-p}} = \gamma \e_{p}  ,
\]
which is the desired result.

When $J \in \PSFAL{p}{T_p}$, observe that $\proj_{p+T_p}(x)=p+\proj_{T_p}(x-p)$ for any $x \in \RR^n$. Thus projecting again the monotone inclusion \eqref{eq:proxmoninc} on $T_p$, we get
\[
\proj_{T_p}(x-p) = \proj_{p+T_p}(x) - p = \gamma \e_{p} ,
\]
whence the claim follows. The linear case is immediate since $p+T_p = T_p$.
\end{proof}

%% file: tex/sec-identification.tex
\section{Activity Identification of Douglas--Rachford}
\label{sec:DRident}

In this section, we present the finite time activity identification of the DR method.

\begin{thm}[Finite activity identification]\label{thm:genident}
Suppose that the DR scheme \eqref{eq:drsgen} is used to create a sequence $(\zk,\xk,\vk)$. Then $(\zk,\xk,\vk)$ converges to $(\zsol,\xsol,\xsol)$, where $\zsol \in \Fix(\BDR)$ and $\xsol$ is a global minimizer of \eqref{eq:minP}. Assume that $J \in \PSF{\xsol}{\MmJ}$ and $G \in \PSF{\vsol}{\MmG}$, and
\begin{equation}\label{eq:conditionnondeg}
\zsol \in \xsol + \gamma\Pa{ \ri\Pa{\partial J(\xsol)}\cap\ri\Pa{- \partial G(\vsol)}}  .
\end{equation} 
Then, 
\begin{enumerate}[label={\rm (\arabic{*})}]
\item The DR scheme has the finite activity identification property, i.e. for all $k$ sufficiently large, $(\xk,\vk) \in \MmJ \times \MmG$.
\item If $G \in \PSFAL{\vsol}{\TG{\vsol}}$, then $\vk \in \vsol+\TG{\vsol}$, and $\TG{\vk}=\TG{\vsol}$ for all $k$ sufficiently large.
\item If $J \in \PSFAL{\xsol}{\TJ{\xsol}}$ , then $\xk \in \xsol+\TJ{\xsol}$, and $\TJ{\xk}=\TJ{\xsol}$ for all $k$ sufficiently large.
\end{enumerate}
\end{thm}
\begin{proof}

Standard arguments using that $\BDR$ is firmly non-expansive allow to show that the iterates $\zk$ converge globally to a fixed point $\zfix \in \Fix(\BDR)$, by interpreting DR as a relaxed Krasnosel'ski\u{\i}-Mann iteration. Moreover, the shadow point $\xsol \eqdef \prox_{\gamma J}(\zfix)$ is a solution of \eqref{eq:minP}, see e.g. \cite{bauschke2011convex}. In turn, using non-expansiveness of $\prox_{\gamma J}$, and as we are in finite dimension, we conclude also that the sequence $\xk$ converges to $\xsol$. This entails that $\vk$ converges to $\vsol$ (by non-expansiveness of $\prox_{\gamma G}$).

Now \eqref{eq:conditionnondeg} is equivalent to
\begin{equation}\label{eq:relativeinclusion}
\qfrac{\zsol - \xsol}{\gamma} \in \ri\Pa{\partial J(\xsol)}  \qandq  \qfrac{\xsol - \zsol}{\gamma} \in \ri\Pa{\partial G(\xsol)}  .
\end{equation}

\begin{enumerate}[label={\rm (\arabic{*})}]
\item 
The update of $\xkp$ and $\vkp$ in \eqref{eq:drsgen} is equivalent to the monotone inclusions
\begin{equation*}
\qfrac{\zkp - \xkp}{\gamma}  \in \partial J(\xkp) \qandq \qfrac{2\xk-\zk - \vkp}{\gamma} \in \partial G(\vkp) ~. 
\end{equation*}
It then follows that
\[
\begin{aligned}
\dist\Pa{\tfrac{\zsol - \xsol}{\gamma}, \partial J(\xkp)} 
\leq~& \qfrac{1}{\gamma}\pa{\norm{\zkp-\zfix} + \norm{\xkp-\xsol}}  \to 0  
\end{aligned}
\]
and
\[
\begin{aligned}
\dist\Pa{\tfrac{\xsol - \zsol}{\gamma}, \partial G(\vkp)} 
\leq~& \qfrac{1}{\gamma}\pa{\norm{\zk-\zfix} + \norm{\xk-\xsol} + \norm{\xk-\vkp}} \\
\leq~& \qfrac{1}{\gamma}\pa{\norm{\zk-\zfix} + 2\norm{\xk-\xsol} + \norm{\vkp-\xsol}} \to 0  .
\end{aligned}
\]
By assumption, $J \in \Gamma_0(\RR^n)$ and $G \in \Gamma_0(\RR^n)$, and thus are sub-differentially continuous at every point in their respective domains \cite[Example~13.30]{rockafellar1998variational}, and in particular at $\xsol$. It then follows that $J(\xk) \to J(\xsol)$ and $G(\vk) \to G(\vsol)$. Altogether, this shows that the conditions of \cite[Theorem~5.3]{hare2004identifying} are fulfilled for $J$ and $G$, and the finite identification claim follows.
%
\item In this case, when $\MmG$ is {\em affine}, then $\MmG = \xsol+\TG{\xsol}$. Since $G$ is partly smooth at $\xsol$ relative to $\xsol+\TG{\xsol}$, the sharpness property holds at all nearby points in $\xsol+\TG{\xsol}$ \cite[Proposition~2.10]{Lewis-PartlySmooth}. Thus for $k$ large enough, i.e. $\vk$ sufficiently close to $\xsol$, we have indeed $\tgtManif{\vk}{\xsol+\TG{\xsol}}=\TG{\xsol}=\TG{\vk}$ as claimed. When $\MmG$ is {\em linear}, then $\TG{\xsol} = \xsol+\TG{\xsol}$, and the result follows easily.
\item Similar to (2). \qedhere
\end{enumerate}
\end{proof}

\begin{rem}\label{rem:fixBPDR}{~}\\\vspace*{-0.5cm}
\begin{enumerate}
\item Condition \eqref{eq:conditionnondeg} can be interpreted as a non-degeneracy assumption, and viewed as a geometric generalization of the strict complementarity of non-linear programming. Such a condition is almost necessary for the finite identification of the partial smoothness active manifolds \cite{HareLewisAlgo}.
\item When the minimizer is unique, using the fixed-point set characterization of DR, it can be shown that condition \eqref{eq:conditionnondeg} is also equivalent to $\zsol \in \ri\Pa{\Fix(\BDR)}$.
\end{enumerate}
\end{rem}

%% file: tex/sec-linear-rate.tex
\section{Local Linear Convergence of Douglas--Rachford}
\label{sec:DRrate}


\subsection{Angles between subspaces}

Let us start with the principal angles and the Friedrichs angle between two subspaces $\Uu$ and $\Vv$, which are crucial for our quantitative analysis of the convergence rates. Without loss of generality, let $1\leq p\eqdef\dim(\Uu) \leq q\eqdef\dim(\Vv) \leq n-1$.

\begin{defn}[Principal angles] 
The principal angles $\theta_k \in [0,\frac{\pi}{2}]$, $k=1,\ldots,p$ between $\Uu$ and $\Vv$ are defined by, with $u_0 = v_0 \eqdef 0$
\begin{align*}
\cos \theta_k \eqdef \dotp{u_k}{v_k} = \max \dotp{u}{v} ~~ s.t. ~~	
& u \in \Uu, v \in \Vv, \norm{u}=1, \norm{v}=1, \\
&\dotp{u}{u_i}=\dotp{v}{v_i}=0, ~ i=0,\ldots,k-1 .
\end{align*}
The principal angles $\theta_k$ are unique with $0 \leq \theta_1 \leq \theta_2 \leq \ldots \leq \theta_p \leq \pi/2$.
\end{defn}

\begin{defn}[Friedrichs angle] 
The Friedrichs angle $\theta_{F} \in ]0,\frac{\pi}{2}]$ between $\Uu$ and $\Vv$ is
\begin{equation*}
\cos \theta_F(\Uu,\Vv) \eqdef \max \dotp{u}{v} ~~s.t.~~
u \in \Uu \cap (\Uu \cap \Vv)^\perp, \norm{u}=1 ,\,
v \in \Vv \cap (\Uu \cap \Vv)^\perp, \norm{v}=1  .
\end{equation*}
\end{defn}

The following relation between the Friedrichs and principal angles is of paramount importance to our analysis, whose proof can be found in \cite[Proposition~3.3]{Bauschke14}.
\begin{lem}[Principal angles and Friedrichs angle]
\label{lem:fapa}
The Friedrichs angle is exactly $\theta_{d+1}$ where $d \eqdef \dim(\Uu \cap \Vv)$. Moreover, $\theta_{F}(\Uu,\Vv) > 0$.
\end{lem}

\begin{rem}
One approach to obtain the principal angles is through the singular value decomposition (SVD). For instance, let $X\in\RR^{n\times p}$ and $Y\in\RR^{n\times q}$ form the orthonormal bases for the subspaces $\Uu$ and $\Vv$ respectively. Let $A \Sigma B^T$ be the SVD of $X^TY \in \RR^{p\times q}$, then $\cos\theta_k = \sigma_k,~ k = 1,2,\ldots,p$ and $\sigma_k$ corresponds to the $k$'th largest singular value in $\Sigma$.
\end{rem}

%
%

\subsection{Partial smoothness and Riemannian gradient and hessian}

Let function $G$ be $C^2$-partly smooth at $\xsol$ relative to a manifold $\MmG$, we denote $\widetilde{G}$ its $C^2$-smooth representative (extension) on $\MmG$. The Riemannian (covariant) gradient of $G$ is the vector field $\nabla_{\MmG} G(\xsol) \allowbreak\in \tgtManif{\MmG}{\xsol}=\TG{\xsol}$,
\[
\dotp{\nabla_{\MmG} G(\xsol)}{h} = \frac{d}{dt}G\Pa{\proj_{\MmG}(\xsol+th)}\big|_{t=0} ,~ \forall h \in \TG{\xsol} ,
\]
where $\proj_{\MmG}$ is the projection operator onto $\MmG$. The Riemannian (covariant) hessian of $G$ is the symmetric linear mapping $\nabla^2_{\MmG} G(\xsol)$ from $\TG{\xsol}$ into itself defined as
\[
\dotp{\nabla^2_{\MmG} G(\xsol)h}{h} = \frac{d^2}{dt^2}G\Pa{\proj_{\MmG}(\xsol+th)}\big|_{t=0} ,~ \forall h \in \TG{\xsol} .
\]
This definition agrees with the usual definition using geodesics or connections. When $\MmG$ is a Riemannian submanifold of $\RR^n$, the Riemannian gradient is also given by
\begin{equation}
\label{eq:riemgrad}
\nabla_{\MmG} G(\xsol) = \proj_{\TG{\xsol}} \nabla \widetilde{G}(\xsol) ,
\end{equation}
and, $\forall h \in \TG{\xsol}$, the Riemannian hessian reads
\[
\nabla_{\MmG}^2 G(\xsol) h = \proj_{\TG{\xsol}} \Pa{\dder\Pa{\nabla_{\MmG} G(\xsol)}[h]} ,
\]
where $\dder$ stands for the directional derivative operator, $\nabla \widetilde{G}(\xsol)$ is the (Euclidean) gradient of $\widetilde{G}$ at $\xsol$. When $\MmG$ is a affine/linear submanifold of $\RR^n$, then obviously $\MmG=\xsol+\TG{\xsol}$, and we get immediately from the definition above that
\begin{equation}
\label{eq:riemhess}
\nabla_{\MmG}^2 G(\xsol) = \proj_{\TG{\xsol}}\nabla^2 \widetilde{G}(\xsol)\proj_{\TG{\xsol}}  ,
\end{equation}
where $\nabla^2 \widetilde{G}(\xsol)$ is the (Euclidean) hessian of $\widetilde{G}$ at $\xsol$.

\begin{lem}
\label{lem:riemgrad}
Let the function $G$ be partly smooth at the point $\xsol$ relative to the manifold $\MmG$. Then given any $x \in \MmG$ near $\xsol$
\[
e_{G}^{x} \eqdef \proj_{\TG{x}}\Pa{\partial G(x)} = \nabla_{\MmG} G(x) = \proj_{\Aff(\partial G(x))}(0) .
\]
Moreover, the Riemannian gradient does not depend on the smooth representation.
\end{lem}

\begin{proof}
The first equalities follow from \cite[Proposition~17]{Daniilidis06} using partial smoothness and local normal sharpness. The last assertion is \cite[Proposition~9]{Daniilidis06}.
\end{proof}

From now on, we assume that the partial smoothness manifolds $\MmG$ and $\MmJ$ are affine/linear, i.e. they are parallel to the corresponding tangent spaces $\TG{\xsol}$ and $\TJ{\xsol}$. Since the latter have a structure of vector space, one can apply the classical Taylor series to the Riemannian gradient presented in Lemma~\ref{lem:riemgrad} and make appear the Riemannian hessian \eqref{eq:riemhess}. We state the result for $G$ and the same claim holds of course for $J$ with proper substitution.

\begin{lem}
\label{lem:riemhesspsd}
Let function $G$ be partly smooth at the point $\xsol$ relative to the affine/linear manifold $\MmG$. For any $h \in \RR^n$, let $\xsol_{h} \eqdef \xsol + \PG h$, then we have
\[
e_{G}^{\xsol_{h}} = e_{G}^{\xsol} + \PG \nabla^2 \widetilde{G}\pa{\xsol} \PG h + o(h) .
\]
Moreover, the Riemannian hessian is semi-positive definite,
\[
\dotp{\PG\nabla^2 \widetilde{G}\pa{\xsol}\PG h}{h} \geq 0 .
\]
\end{lem}

\begin{proof}
The first assertion is clear from the discussion above. We now prove the second claim. As $G$ is a proper lsc convex function, $\partial G$ is a maximal monotone operator. Thus, $\forall t > 0$,
\begin{align*}
0 \leq 	\dotp{t^{-1}\pa{v - u}}{\PG h} 	
&= \dotp{t^{-1}\PG\pa{v - u}}{h} \quad \forall u \in \partial G(\xsol) \qandq  v \in \partial G(\xsol+t \PG h) \\
\scriptsize{\text{($\MmG$ is affine/linear})\quad}
&= \dotp{t^{-1}\Pa{\proj_{\TG{\xsol_{h}}}v - \PG u}}{h} \quad \forall u \in \partial G(\xsol) \qandq  v \in \partial G(\xsol+t\PG h) \\
\scriptsize{\text{(By definition)}\quad}
&= \dotp{t^{-1}\pa{e_{G}^{\xsol_{h}} - e_{G}^{\xsol}}}{h} 	\\
\scriptsize{\text{(Lemma~\ref{lem:riemgrad})}\quad}
&= \dotp{t^{-1}\pa{\nabla_{\MmG} G(\xsol+t\PG h)-\nabla_{\MmG} G(\xsol)}}{h} \\
\scriptsize{\text{(\eqref{eq:riemgrad} and $\MmG$ is affine/linear)}\quad}
&= \dotp{t^{-1}\PG\pa{\nabla \widetilde{G}(\xsol+t\PG h)-\nabla \widetilde{G}(\xsol)}}{h} .
\end{align*}
Passing to the limit as $t \to 0$ leads to the desired result.
\end{proof}

\subsection{Convergence rates of a fixed-point matrix}
We now establish the convergence rates of a matrix that plays a fundamental role in the DR algorithm.

Let
\begin{equation}\label{eq:Pk-P}
P = \gamma \PG \nabla^2 \widetilde{G}\pa{\xsol} \PG  \qandq 
Q = \gamma \PJ \nabla^2 \widetilde{J}\pa{\xsol} \PJ .
\end{equation}
Owing to Lemma~\ref{lem:riemhesspsd}, $\Id+P$ and $\Id+Q$ are symmetric positive definite, hence invertible. We then write their inverses as
\[
U = \pa{\Id+P}^{-1} \qandq V = \pa{\Id+Q}^{-1} .
\]
Define the matrix
\begin{equation}\label{eq:mtx-M}
\begin{aligned}
M 
&= \Id + 2\PG U \PG \PJ V\PJ - \PG U \PG - \PJ V \PJ  \\
&= \qfrac12\Id + \PG U \PG\pa{2 \PJ V\PJ - \Id} - \qfrac12\pa{2\PJ V\PJ-\Id}  \\
&= \qfrac12\Id + \qfrac12\pa{2\PG U \PG-\Id}\pa{2 \PJ V\PJ - \Id}   ,
\end{aligned}
\end{equation}
and the one parameterized by $\lambda_k \in ]0,2[$,
\[
M_{\lambda_k} = (1-\lambda_k)\Id + \lambda_k M  .
\]
Obviously, given any $\lambda \in ]0,2[$, we have
\[
M_{\lambda_k} - M_{\lambda} = -(\lambda_k-\lambda)\pa{\Id - M}  .
\]
To lighten the notation, we denote
\[
\BG \eqdef \PG U \PG \qandq \CJ \eqdef \PJ V \PJ .
\]

Our proofs will hinge on the following key lemma which characterizes the convergence behaviour of $M_{\lambda}$. We denote $\SJ{\xsol}=\pa{\TJ{\xsol}}^\perp$ and similarly for $\SG{\xsol}$.

\begin{lem}
\label{lem:Mpow}
Suppose that $\lambda \in ]0,2[$, then,
\begin{enumerate}[label={\rm (\arabic{*})}]
\item $M_{\lambda}$ is convergent to 
\[
M^\infty=\proj_{\ker\pa{\BG\pa{\Id-\CJ}+\pa{\Id-\BG}\CJ}},
\] 
and we have
\[
\forall k \in \NN,~ M_{\lambda}^k - M^\infty = \pa{M_{\lambda} - M^\infty}^k
\qandq
\rho\pa{M_{\lambda}-M^\infty} < 1.
\]
In particular, if 
\begin{equation}\label{eq:condMinfz}
\begin{cases}
\TJ{\xsol} \cap \TG{\vsol}=\ens{0}, \\
\Im\Pa{\Id-\CJ} \cap \SG{\vsol}=\ens{0} \qandq \\
\Im\Pa{\Id-\BG} \cap \TG{\xsol}=\ens{0}  ,
\end{cases}
\end{equation}
then $M^\infty=0$.
\item Given any $\rho \in ]\rho\pa{M_{\lambda}-M^\infty}, 1[$, there is $K$ large enough such that for all $k \geq K$,
\[
\norm{M_{\lambda}^k-M^\infty} = O(\rho^k) ~.
\]
\item If, moreover, $G \in \PSFLS{\xsol}{\TG{\xsol}}$ and $J \in \PSFLS{\xsol}{\TJ{\xsol}}$, then $M_{\lambda}$ converges to $\proj_{\pa{\TJ{\xsol} \cap \TG{\vsol}} \oplus \pa{\SJ{\xsol} \cap \SG{\vsol}}}$ with the optimal rate
\[
\sqrt{\pa{1-\lambda}^2+\lambda\pa{2-\lambda}\cos^2\theta_F\Pa{\TJ{\xsol},T_{\xsol}^{G}}} < 1 .
\]
In particular, if $\TJ{\xsol} \cap \TG{\vsol}=\SJ{\xsol} \cap \SG{\vsol}=\ens{0}$, then $M_{\lambda}$ converges to $0$ with the optimal rate 
\[
\sqrt{\pa{1-\lambda}^2+\lambda\pa{2-\lambda}\cos^2\theta_1\Pa{\TJ{\xsol},T_{\xsol}^{G}}} < 1 .
\]
\end{enumerate}
\end{lem}

\begin{proof}
$~$\vspace*{-0.25cm}
\begin{enumerate}[label={\rm (\arabic*)}]
\item Since $U$ (resp. $V$) is linear, symmetric, and has eigenvalues in $]0,1]$, it is firmly non-expansive \cite[Corollary~4.3(ii)]{bauschke2011convex}. 
It then follows from \cite[Example~4.7]{bauschke2011convex} that $\BG$ and $\CJ$ are firmly non-expansive.
Therefore, we get that $M$ is {\rm firmly non-expansive} \cite[Proposition~4.21(i)-(ii)]{bauschke2011convex}, or equivalently that $M_{\lambda}$ is $\frac{\lambda}{2}$-averaged \cite[Corollary~4.29]{bauschke2011convex}. We then conclude from e.g.~\cite[Proposition~5.15]{bauschke2011convex} that $M_{\lambda}$ and $M$ are convergent, and their limit is $M_{\lambda}^{\infty} = \proj_{\Fix M_{\lambda}}=\proj_{\Fix M}=M^\infty$ \cite[Corollary~2.7(ii)]{Bauschke14}. Moreover, $M_{\lambda}^k - M^\infty = \pa{M_{\lambda} - M^\infty}^k$, $\forall k \in \NN$, and $\rho\pa{M_{\lambda}-M^\infty} < 1$ by \cite[Theorem~2.12]{Bauschke14}. It is also immediate to see that
\[
\Fix M = \ker\Pa{\BG\pa{\Id-\CJ}+\pa{\Id-\BG}\CJ} ~.
\]
Observe that
\begin{gather*}
\Im(\CJ) \subseteq \TJ{\xsol} \qandq \Im(\BG) \subseteq \TG{\xsol} \\
\ker\Pa{\Id-\BG} \subseteq \TG{\xsol} \qandq \ker\pa{\BG } = \SG{\vsol} \\
\Im\Pa{\pa{\Id-\BG}\CJ} \subseteq \Im\Pa{\Id-\BG} \qandq \Im\Pa{\BG\pa{\Id-\CJ}} \subseteq \TG{\xsol}  ,
\end{gather*}
where we used the fact that $U$ and $V$ are positive definite.
Therefore, $M_{\lambda}^{\infty}=0$, if and only if, $\Fix M=\ens{0}$, and for this to hold true, it is sufficient that
\[
\begin{cases}
\Im\pa{\CJ} \cap \ker\pa{\Id-\BG} \subseteq \TJ{\xsol} \cap \TG{\vsol}=\ens{0}, \\
\Im\pa{\Id-\CJ} \cap \ker\pa{\BG } = \Im\pa{\Id-\CJ} \cap \SG{\vsol}=\ens{0} \qandq \\
\Im\Pa{\pa{\Id-\BG}\CJ} \cap \Im\Pa{\BG\pa{\Id-\CJ}} \subseteq \Im\pa{\Id-\BG} \cap \TG{\xsol}=\ens{0} .
\end{cases}
\]
\item The proof of this statement is classical using the spectral radius formula, see e.g. \cite[Theorem~2.12(i)]{Bauschke14}.
\item In this case, we have $U = V = \Id$. In turn, $\BG=\PG$ and $\CJ=\PJ$, which yields 
\[
M = \Id + 2\PG \PJ - \PG - \PJ = \PG \PJ + \PGs\PJs  ,
\]
which is normal, and so is $M_\lambda$. From \cite[Proposition~3.6(i)]{BauschkeDR14}, we get that $\Fix M = \pa{\TJ{\xsol} \cap \TG{\vsol}} \oplus \pa{\SJ{\xsol} \cap \SG{\vsol}}$. Thus, combining normality, statement (i) and \cite[Theorem~2.16]{Bauschke14} we get that
\[
\norm{M_{\lambda}^{k+1-K}-M^{\infty}} = \norm{M_{\lambda}-M^{\infty}}^{k+1-K}
\]
and $\norm{M_{\lambda}-M^{\infty}}$ is the optimal convergence rate of $M_{\lambda}$. Using together \cite[Proposition 3.3]{Bauschke14} and arguments similar to those of the proof of \cite[Theorem~3.10(ii)]{BauschkeDR14} (see also \cite[Theorem~4.1(ii)]{Bauschke14}), we get indeed that
\[
\norm{M_{\lambda}-M^{\infty}} = \sqrt{\pa{1-\lambda}^2+\lambda\pa{2-\lambda}\cos^2\theta_F\Pa{\TJ{\xsol},T_{\xsol}^{G}}} .
\]
The special case is immediate. This concludes the proof. \qedhere
\end{enumerate}
\end{proof}

\subsection{Main result}

We are now in position to present the local linear convergence properties of DR. 

Denote 
\[
\delta_k \eqdef o\pa{\norm{\vkp-\vsol}} + o\pa{\norm{\xk-\xsol}} + o\pa{\norm{\zk-\zsol}} \qandq \Delta_{k,K} \eqdef \msum_{j=K}^k \delta_j  .
\]

\begin{thm}\label{thm:genlinear_rate}
Suppose that the DR scheme \eqref{eq:drsgen} is used with $\lambda_k \to \lambda \in ]0,2[$ to create a sequence $(\zk,\xk,\vk) \to (\zsol,\xsol,\xsol)$ such that  $J \in \PSFAL{\xsol}{\TJ{\xsol}}$ and $G \in \PSFAL{\vsol}{\TG{\vsol}}$, and \eqref{eq:conditionnondeg} holds. 
Then, 
\begin{enumerate}[label={\rm (\arabic{*})}]
\item Given any $\rho \in ]\rho\pa{M_{\lambda}-M^\infty}, 1[$, there is $K$ large enough such that for all $k \geq K$,
\begin{equation*}
\norm{\pa{\zk - \zfix} - {M^\infty}\pa{z^{K} - \zfix + \Delta_{k,K}}} = O\pa{\rho^{k}} .
\end{equation*}
In particular, if condition \eqref{eq:condMinfz} holds, then given any $\rho \in ]\rho\pa{M_{\lambda}-M^\infty}, 1[$, there is $K$ large enough such that for all $k \geq K$,
\begin{equation*}
\norm{\zk - \zfix} = O\pa{\rho^{k}} .
\end{equation*}
\item Assume moreover that $J \in \PSFALS{\xsol}{\TJ{\xsol}}$ and $G \in \PSFALS{\vsol}{\TG{\vsol}}$, and $\lambda_k \equiv \lambda \in ]0, 2[$. Then, there exists $K > 0$ such that for all $k \geq K$,
\begin{equation}\label{eq:rate}
\begin{aligned}
\norm{\pa{\zk - \zfix} - \proj_{\pa{\TJ{\xsol} \cap \TG{\vsol}} \oplus \pa{\SJ{\xsol} \cap \SG{\vsol}}}\pa{z^{K} - \zfix}}
&\leq \rho^{k-K} \norm{\pa{\Id-\proj_{\pa{\TJ{\xsol} \cap \TG{\vsol}} \oplus \pa{\SJ{\xsol} \cap \SG{\vsol}}}}\pa{z^{K} - \zfix}} \\
&\leq \rho^{k-K} \norm{z^{K} - \zfix},
\end{aligned}
\end{equation}
where $\rho = \sqrt{\pa{1-\lambda}^2+\lambda\pa{2-\lambda}\cos^2\theta_F\Pa{\TJ{\xsol},T_{\xsol}^{G}}} \in [0,1[$ is the optimal convergence rate.

In particular, if $\TJ{\xsol} \cap \TG{\vsol}=\SJ{\xsol} \cap \SG{\vsol}=\ens{0}$, then $\zk$ converges locally linearly to $\zsol$ with the optimal rate $\sqrt{\pa{1-\lambda}^2+\lambda\pa{2-\lambda}\cos^2\theta_1\Pa{\TJ{\xsol},T_{\xsol}^{G}}}$.
\end{enumerate}
\end{thm}
\begin{rem}
It can be observed that for the last statement, the best rate is obtained for $\lambda = 1$. This has been also pointed out in \cite{DemanetZhang13} for basis pursuit. This assertion is however only valid for the local convergence behaviour and does not mean in general that the DR will be globally faster for $\lambda_k \equiv 1$. Note also that the above result can be straightforwardly generalized to the case of varying $\lambda_k$.
\end{rem}
\begin{proof}
Since by assumption $\lambda_k \to \lambda \in ]0, 2[$ and $\Id - M_{\lambda}$ is non-expansive by Lemma~\ref{lem:Mpow}, we have
\[
\lim_{k \to \infty} \qfrac{\norm{(M_{\lambda_k}-M_{\lambda})(\zk-\zsol)}}{\norm{\zk-\zsol}}
= \lim_{k \to \infty} \qfrac{\abs{\lambda_k - \lambda}\norm{(\Id-M)(\zk-\zsol)}}{\norm{\zk-\zsol}}
\leq \lim_{k \to \infty} \abs{\lambda_k - \lambda} = 0  ,
\]
which means that ${(M_{\lambda_k}-M_{\lambda})(\zk-\zsol)} = o(\norm{\zk-\zsol})$ when $k$ is large enough.
\begin{enumerate}[label={\rm (\arabic{*})}]
\item We have
\[
\left\{
\begin{aligned}
\vkp &= \prox_{\gamma G}\pa{2\xk - \zk}  , \\
\xkp &= \prox_{\gamma J} \zkp  ,
\end{aligned}
\right.
~\Longleftrightarrow~
\left\{
\begin{aligned}
\pa{2\xk - \zk}  - \vkp &\in \gamma \partial G(\vkp) , \\
\zkp - \xkp &\in \gamma \partial J(\xkp)  ,
\end{aligned}
\right.
\]
and
\begin{equation*}
\zkp 
= \pa{1-\lambda_k}\zk+ \lambda_k \pa{\zk + \vkp - \xk} .
\end{equation*}
Thus, by Proposition~\ref{prop:proxps} we have,
\[
\PJ\pa{ \zk  - \xk } = \gamma e_{J}^{k} \qandq
\PJ\pa{ \zsol - \xsol } = \gamma e_{J}^{\star} ~,
\]
which, after using Theorem~\ref{thm:genident}(2)-(3) and Lemma~\ref{lem:riemhesspsd}, leads to
\[
\begin{aligned}
\pa{\xk-\xsol} + \gamma\pa{e_{J}^{k}-e_{J}^{\star}}
&= \pa{\Id+Q}\pa{\xk-\xsol} + o\pa{\norm{\xk-\xsol}} \\
&= \PJ\pa{\zk-\zsol}  .
\end{aligned}
\]
This yields
\[
\xk-\xsol
= \PJ(\xk-\xsol)
= \PJ V \PJ\pa{\zk-\zsol} + o\pa{\norm{\xk-\xsol}}.
\]
Similarly for $\vkp$, we have
\[
\begin{aligned}
\PG\pa{\vkp-\vsol} + \gamma\pa{e_{G}^{k+1} - e_{G}^{\star}}
&= 2\PG\pa{\xk - \xsol} - \PG\pa{\zk-\zsol}  \\
\Rightarrow~\pa{\Id + P}\pa{\vkp-\vsol} + o\pa{\norm{\vkp-\vsol}}
&= 2\PG\pa{\xk - \xsol} - \PG\pa{\zk-\zsol}  \\
\Rightarrow \qquad\qquad \pa{\vkp-\vsol} + o\pa{\norm{\vkp-\vsol}}
&= 2\PG U \PG\pa{\xk - \xsol} - \PG U \PG\pa{\zk-\zsol} .
\end{aligned}
\]
Therefore
\begin{align*}
\vkp-\vsol
&= 2\PG U \PG \PJ V\PJ\pa{\zk - \zsol} - \PG U \PG\pa{\zk-\zsol} \\
& \qquad + o\pa{\norm{\vkp-\vsol}} + o\pa{\norm{\xk-\xsol}} .
\end{align*}
For the fixed point iterates $\zk$, we have
\[
\begin{aligned}
&\pa{\zk + \vkp - \xk} - \pa{\zsol + \vsol - \xsol}  \\
=~& \pa{\zk - \zsol} + \pa{\vkp - \vsol} - \pa{\xk - \xsol}  \\
=~& \Pa{\Id + 2\PG U \PG \PJ V\PJ - \PG U \PG - \PJ V \PJ}\pa{\zk - \zsol} \\&\qquad + o\pa{\norm{\vkp-\vsol}} + o\pa{\norm{\xk-\xsol}}  \\
=~& M\pa{\zk - \zsol} + o\pa{\norm{\vkp-\vsol}} + o\pa{\norm{\xk-\xsol}} ~,
\end{aligned}
\]
which in turn yields
\[
\begin{aligned}
\zkp-\zsol 
&= M_{\lambda_k}\pa{\zk - \zsol} + o\pa{\norm{\vkp-\vsol}} + o\pa{\norm{\xk-\xsol}} \\
&= M_{\lambda}\pa{\zk - \zsol} + \Pa{M_{\lambda_k}-M_{\lambda}}\pa{\zk - \zsol} + o\pa{\norm{\vkp-\vsol}} + o\pa{\norm{\xk-\xsol}}  \\
&= M_{\lambda}\pa{\zk - \zsol} + \delta_k .
\end{aligned}
\]
From Lemma~\ref{lem:Mpow}(1), $M_{\lambda}$ is indeed convergent to $M^\infty$ as given there. We then have
\[
\begin{aligned}
&\pa{\zkp-\zsol} - M^{\infty}\pa{z^K-\zsol+\Delta_{k,K}}  \\
=~& \pa{M_{\lambda}^{k+1-K} - M^{\infty}}\pa{z^K-\zsol} + \msum_{j=K}^{k}\pa{M_{\lambda}^{k-j} - M^{\infty}} \delta_j \\
=~& \pa{M_{\lambda} - M^{\infty}}^{k+1-K} \pa{z^K-\zsol} + \msum_{j=K}^{k}\pa{M_{\lambda} - M^{\infty}}^{k-j} \delta_j ,
\end{aligned}
\]
whence we get
\[
\begin{aligned}
& \norm{\pa{\zkp-\zsol} - M^{\infty}\pa{z^K-\zsol+\Delta_{k,K}}}  \\
\leq~& \norm{\pa{M_{\lambda} - M^{\infty}}^{k+1-K}} \norm{z^K-\zsol} + \msum_{j=K}^{k}\norm{\pa{M_{\lambda} - M^{\infty}}^{k-j}} \norm{\delta_j} ~.
\end{aligned}
\]
We also observe that
\begin{align}
\norm{\delta_j}	&= o\Ppa{\norm{v^{j+1}-\xsol}+\norm{x^j-\xsol}+\norm{z^j-\zsol}} \nonumber\\
				&= o\Ppa{3\norm{z^j-\zsol}+\norm{z^j-\zsol}+\norm{z^j-\zsol}} \nonumber\\
\label{eq:odelj}
				&= o\Ppa{\norm{z^j-\zsol}} ,
\end{align}
where we used non-expansiveness of the proximity operator. Thus, using Lemma~\ref{lem:Mpow}(2) and \eqref{eq:odelj} proves the local linear convergence claim.

Under condition~\eqref{eq:condMinfz}, we have $M^\infty=0$, and the claim follows.

\item Now, we have $\delta_k=0$, $\forall k \in \NN$, and $M_\lambda$ is normal. It then follows from Lemma~\ref{lem:Mpow}(3) that
\begin{align*}
\norm{\pa{\zkp - \zfix} - M^\infty(z^{K} - \zfix)}
&= \norm{\pa{M_{\lambda}^{k+1-K}-M^\infty}\pa{z^{K} - \zfix}} \\
&= \norm{\pa{M_{\lambda}^{k+1-K}-M^\infty}\pa{\Id-M^\infty}(z^{K} - \zfix)} \\
&\leq \norm{M_{\lambda}^{k+1-K}-M^\infty}\norm{\pa{\Id-M^\infty}(z^{K} - \zfix)} \\
&= \norm{M_{\lambda}-M^\infty}^{k+1-K}\norm{\pa{\Id-M^\infty}(z^{K} - \zfix)} \\
&=\rho^{k+1-K} \norm{\pa{\Id-M^\infty}(z^{K} - \zfix)} \\
&\leq \rho^{k+1-K} \norm{z^{K} - \zfix},
\end{align*}
where $\rho$ is the optimal rate in Lemma~\ref{lem:Mpow}(3), and we have used the fact that $M_{\lambda}^kM_{\lambda}^\infty=M_{\lambda}^k\proj_{\Fix M_{\lambda}}=\proj_{\Fix M_{\lambda}}$, and $\Id-M_{\lambda}^{\infty}$ is an orthogonal projector, hence non-expansive.

The particular case is immediate. This concludes the proof.  \qedhere
\end{enumerate}
\end{proof}

%% file: tex/sec-m-functions.tex
\section{Sum of more than two functions}
\label{sec:DRmfunc}

We now want to tackle the problem of solving
\begin{equation}\label{eq:minPsum}
\min_{x \in \RR^n} \msum_{i=1}^m J_i(x)  ,
\end{equation}
where each $J_i \in \Gamma_0(\RR^n)$. We assume that all the relative interiors of their domains have a non-empty intersection, that the set of minimizers is non-empty, and that these functions are simple.

In fact, problem~\eqref{eq:minPsum} can be equivalently reformulated as \eqref{eq:minP} in a product space, see e.g. \cite{CombettesPesquet08,gfb2011}. Let $\bcH=\underset{\text{$m$ times}}{\underbrace{\RR^n \times \cdots \times \RR^n}}$ endowed with the scalar inner-product and norm 
\[
\forall \bmx, \bmy \in \bcH,~\bprod{\bmx}{\bmy} =\msum_{i=1}^m \dotp{x_i}{y_i},~\bnorm{\bmx} = \sqrt{\msum_{i=1}^m\norm{x_i}^2}.
\]
Let $\bcS = \enscond{\bmx=(x_i)_i\in\bcH}{x_1=\dotsm=x_m}$ and its orthogonal complement $\bcS^\perp = \Ba{\bmx=(x_i)_i\in\bcH: \sum_{i=1}^m x_i=0}$. Now define the canonical isometry,
\[
\bmC: \RR^n \to \bcS,~x\mapsto(x,\dotsm,x) ,
\]
then we have $\proj_{\bcS}(\bmz) = \bmC\!\Ppa{\tfrac{1}{m}\sum_{i=1}^m z_i}$.

Problem~\eqref{eq:minPsum} is now equivalent to
\begin{equation}\label{eq:minPsumprod}
\min_{\bmx \in \bcH} \bmJ(\bmx) + \bmG(\bmx) , ~~\mathrm{where}~~ \bmJ(\bmx) = \msum_{i=1}^m J_i(x_i) ~~\mathrm{and}~~ \bmG(\bmx) = \iota_{\bcS}(\bmx)  ,
\end{equation}
which has the form \eqref{eq:minP} on $\bcH$.

Obviously, $\bmJ$ is separable and therefore,
\[
\prox_{\gamma \bmJ}(\bmx) = \Pa{\prox_{\gamma J_i}(x_i)}_i  .
\]
Let $\bmx^\star = \bmC(\xsol)$. Clearly, $\bmG$ is polyhedral, hence partly smooth relative to $\bcS$ with $\bmT_{\bmx^\star}^{\bmG} =\bcS$.
Suppose that $J_i \in \PSF{\xsol}{\MmJi{\xsol}}$ for each $i$.
Denote $\TmJ{\bmx^\star} = \bigtimes_i \TJi{\xsol}$ and $\SmJ{\bmx^\star} = (\TmJ{\bmx^\star})^\perp = \bigtimes_i (\TJi{\xsol})^\perp$.
Similarly to \eqref{eq:Pk-P}, define 
\[
\bmQ = \gamma \bPJ \nabla^2 \widetilde{\bmJ}\pa{\xsol} \bPJ  \qandq
\bmV = \pa{\bId + \bmQ}^{-1}  ,
\] 
where $\widetilde{\bmJ}(\bmx) \eqdef \msum_{i=1}^m \widetilde{J}_i(x_i)$ is the smooth representation of $\bmJ$, and $\bId$ is the identity operatror on $\bcH$. Now we can provide the product space form of \eqref{eq:mtx-M}, where we recall that $\bmG$ is polyhedral,
\begin{equation}\label{eq:mtx-Mprodspace}
\begin{aligned}
\bmM 
&= \bId + 2\bPG  \bPJ \bmV \bPJ - \bPG  - \bPJ \bmV \bPJ  \\
&= \qfrac12\bId + \bPG \pa{2 \bPJ \bmV \bPJ - \bId} - \qfrac12\pa{2\bPJ \bmV \bPJ-\bId}  \\
&= \qfrac12\bId + \qfrac12\pa{2\bPG -\bId}\pa{2 \bPJ \bmV \bPJ - \bId}   ,
\end{aligned}
\end{equation}
and $\bmM_{\lambda} = (1-\lambda)\bId + \lambda \bmM$. Owing to Lemma \ref{lem:Mpow}, we have
\[
\bmM^\infty=\proj_{\ker\pa{\bPG\pa{\Id-\bPJ \bmV \bPJ}+\pa{\Id-\bPG}\bPJ \bmV \bPJ}}  .
\]

For the sake of simplicity, we fix $\lambda_k \equiv \lambda \in ]0, 2[$, and define
\[
\boldsymbol{\delta}_k \eqdef o\pa{\norm{\bmx^k-\bmx^\star}}  \qandq \boldsymbol{\Delta}_{k,K} \eqdef \msum_{j=K}^k \boldsymbol{\delta}_j  .
\]
hence we have the following result.

\begin{cor}
\label{cor:bpDRprodsum}
Suppose that the DR scheme is used to solve \eqref{eq:minPsumprod} and creates a sequence $(\bmz^k,\bmx^k,\bmv^k)$. Then $(\bmz^k,\bmx^k,\bmv^k)$ converges to $(\bmz^\star,\bmx^\star,\bmx^\star)$, and $\xsol$ is a minimizer of~\eqref{eq:minPsum}. Suppose that $J_i \in \PSF{\xsol}{\MmJi}$ and
\begin{equation}\label{eq:prodsumconditionnondeg}
\bmz^\star \in \bmx^\star + \gamma \ri\Pa{\partial \bmJ(\bmx^\star)}\cap \bcS^\bot  .
\end{equation} 
Then, 
\begin{enumerate}[label={\rm (\arabic{*})}]
\item the DR scheme has the finite activity identification property, i.e. for all $k$ sufficiently large, $\bmx^k \in \bigtimes_i \MmJi$.
\item Suppose that $J_i \in \PSFAL{\xsol}{\TJi{\xsol}}$, then given any $\rho \in ]\rho\pa{\bmM_{\lambda}-\bmM^\infty}, 1[$, there is $K$ large enough such that for all $k \geq K$,
\begin{equation*}
\norm{\pa{\bmz^k - \bmz^\star} - {M^\infty}\pa{\bmz^{K} - \bmz^\star + \boldsymbol{\Delta}_{k,K}}} = O\pa{\rho^{k}} .
\end{equation*}
%
\item Assume that $J_i \in \PSFLS{\xsol}{\TJi{\xsol}}$, then, there exists $K > 0$ such that for all $k \geq K$,
\begin{equation*}
\norm{\pa{\bmz^k - \bmz^\star} - \proj_{\pa{\TmJ{\xsol} \cap \bcS} \oplus \pa{\SmJ{\xsol} \cap \bcS^\perp}}\pa{\bmz^K - \bmz^\star}} \leq \rho^{k-K} \norm{\bmz^{K} - \bmz^\star}  ,
\end{equation*}
with $\rho = \sqrt{\pa{1-\lambda}^2+\lambda\pa{2-\lambda}\cos^2\theta_F\Pa{\TmJ{\xsol},\bcS}} \in [0,1[$, and thus, $\bmz^k - \bmz^\star$ converges locally linearly to $\proj_{\pa{\TmJ{\xsol} \cap \bcS} \oplus \pa{\SmJ{\xsol} \cap \bcS^\perp}}(\bmz^{K} - \bmz^\star)$ at the optimal rate $\rho$.
\end{enumerate}
\end{cor} 
\begin{proof}
{~}\\\vspace*{-0.5cm}
\begin{enumerate}[label={\rm (\arabic{*})}]
\item By the separability rule, $\bmJ \in \PSF{\bmx^\star}{\bigtimes_{i} \MmJi_{\xsol}}$, see \cite[Proposition~4.5]{Lewis-PartlySmooth}. 
We also have $\partial \bmG(\bmx^\star) = N_{\bcS}(\bmx^\star)=\bcS^\bot$. 
Then \eqref{eq:prodsumconditionnondeg} is simply a specialization of condition~\eqref{eq:conditionnondeg} to problem~\eqref{eq:minPsumprod}. The claim then follows from Theorem~\ref{thm:genident}(1).
\item This is a direct consequence of Theorem~\ref{thm:genlinear_rate}(2). 
\item This is a direct consequence of Theorem~\ref{thm:genlinear_rate}(3).  \qedhere
\end{enumerate}
\end{proof}

%% file: tex/sec-experiment.tex
\section{Numerical experiments}
\label{sec:DRapp}

Here, we illustrate our theoretical results on several concrete examples. This section is by no means exhaustive, and we only focus on the problems that we consider as representative in variational signal/image processing.

\paragraph{Affinely-constrained Minimization}
Let us now consider the affine-constrained minimization problem
\begin{equation}\label{eq:minbp}
  \min_{x \in \RR^n} J(x) \qsubjq y = A x  ,
\end{equation}
where $A \in \RR^{m \times n}$.
We assume that the problem is feasible, i.e. the observation $y \in \Im(A)$. By identifying $G$ with the indicator function of the affine constraint, it is immediate to see that $G=\iota_{\Ker(A)}(\cdot)$, which is polyhedral and is simple (i.e. the corresponding projector can be easily computed). 

Problem~\eqref{eq:minbp} is of important interest in various areas, including signal and image processing to find regularized solutions to linear equations. Typically, $\J$ is a regularization term intended to promote solutions conforming to some notion of simplicity/low-dimensional structure. One can think of instance of the active area of compressed sensing (CS) and sparse recovery. 

We here solve~\eqref{eq:minbp} with $J$ being either $\ell_1$ (Lasso), $\ell_\infty$ (anti-sparsity), and $\ell_{1,2}$-norm (group Lasso).
For all these cases, $J \in \Gamma_0(\RR^n)$, is simple, and is partly smooth relative to a subspace $T_{\xsol}^J$ that can be easily computed, see e.g.~\cite{vaiter2013model}. In fact, in the first two examples, $J$ are polyhedral while $\ell_{1,2}$-norm is not. In these experiments, $A$ is drawn randomly from the standard Gaussian ensemble, i.e. CS scenario, with the following settings:
\begin{enumerate}[label = (\alph{*}), leftmargin=2.5em]
\item {$\ell_1$-norm:} $m=32$ and $n=128$, $x_0$ is $8$-sparse;
\item {$\ell_{1,2}$-norm:} $m=32$ and $n=128$, $x_0$ has $3$ non-zero blocks of size $4$;
\item {$\ell_\infty$-norm:} $m=120$ and $n=128$, $x_0$ has $10$ saturating entries;
\end{enumerate}
For each setting, the number of measurements is sufficiently large so that one can prove that the minimizer $\xsol$ is unique, and in particular that $\ker(A) \cap T_{\xsol} = \ens{0}$ (with high probability). We also checked that $\Im(A^T) \cap S_{\xsol} = \ens{0}$, which is in this case equivalent to uniqueness of the fixed point (see Remark~\ref{rem:fixBPDR}(ii)). Thus \eqref{eq:conditionnondeg} is obviously fulfilled, and the second part of Theorem~\ref{thm:genlinear_rate} applies.

Figure~\ref{fig:DR1}(a)-(c) displays the global profile of $\norm{\zk - \zfix}$ as a function of $k$, and the starting point of the solid line is the iteration number at which the partial smooth manifolds (here subspaces) are identified.
One can easily see that for $\ell_1$, $\ell_\infty$-norms,  the linear convergence behaviour and that our rate estimate is indeed optimal. For the case of $\ell_{1,2}$-norm, though not optimal, our estimate is rather tight.

\begin{figure}[!ht]
\centering
\subfigure[CS $\ell_1$-norm]{\includegraphics[width=0.32\textwidth]{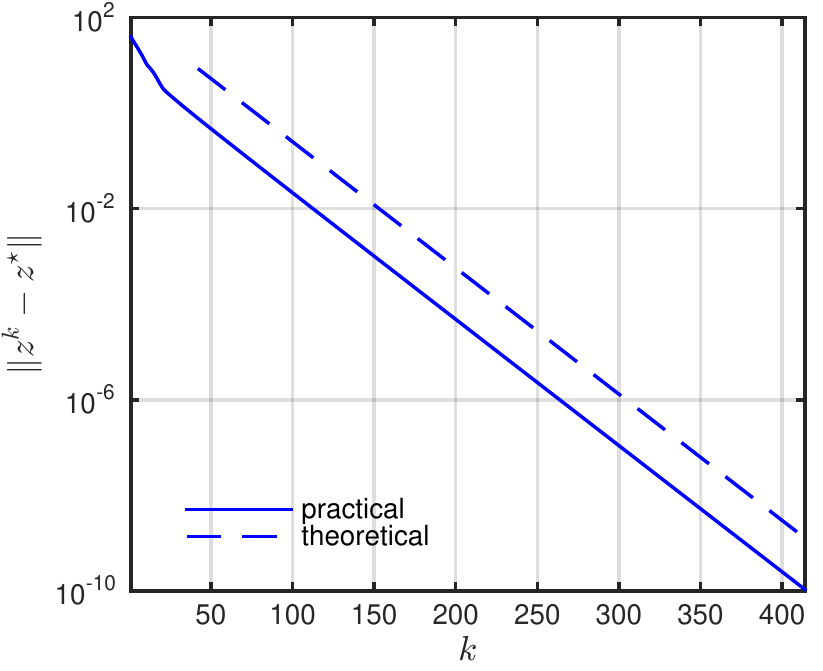}}
\subfigure[CS $\ell_{1,2}$-norm]{\includegraphics[width=0.32\textwidth]{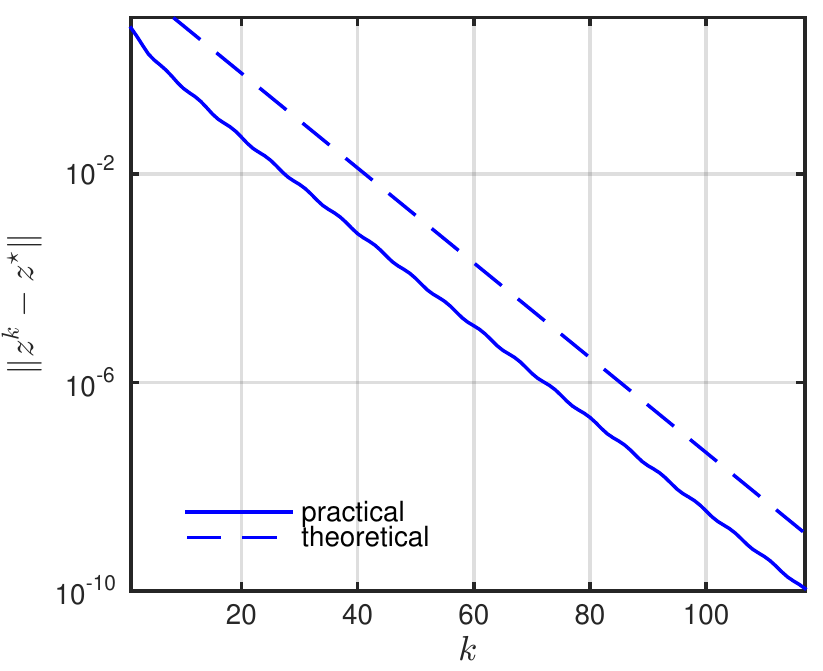}} 
\subfigure[CS $\ell_\infty$-norm]{\includegraphics[width=0.32\textwidth]{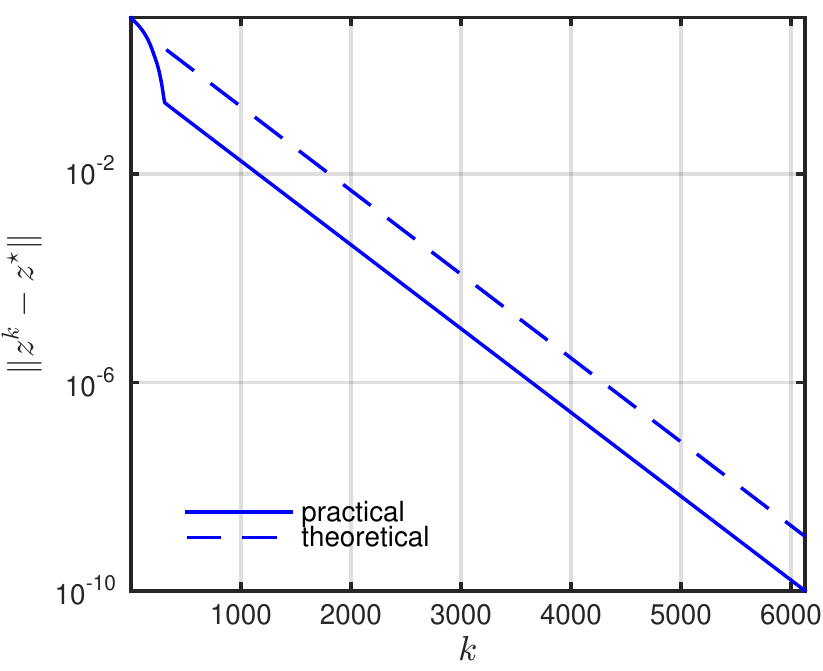}}  \\
\subfigure[TV image inpainting]{\includegraphics[width=0.32\textwidth]{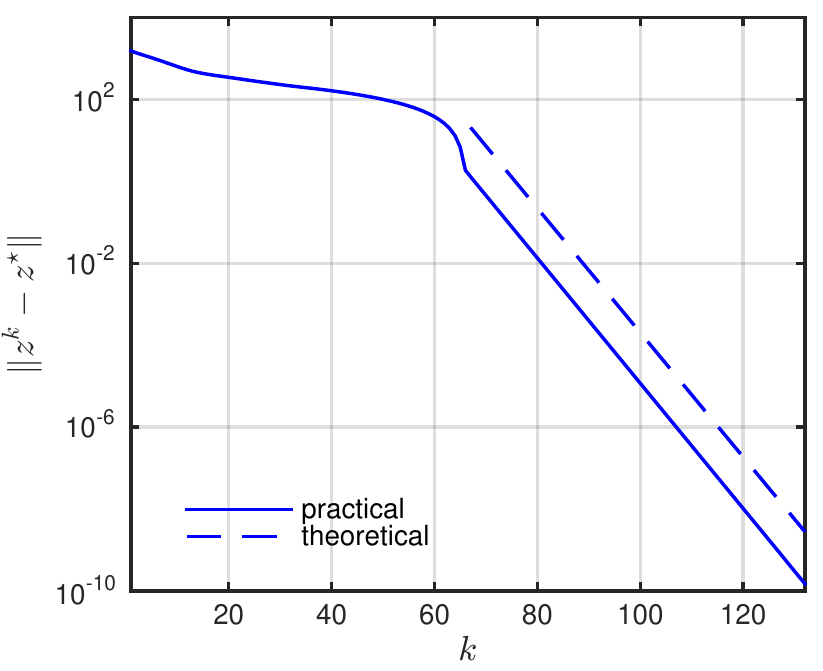}}
\subfigure[Uniform noise removal]{\includegraphics[width=0.32\textwidth]{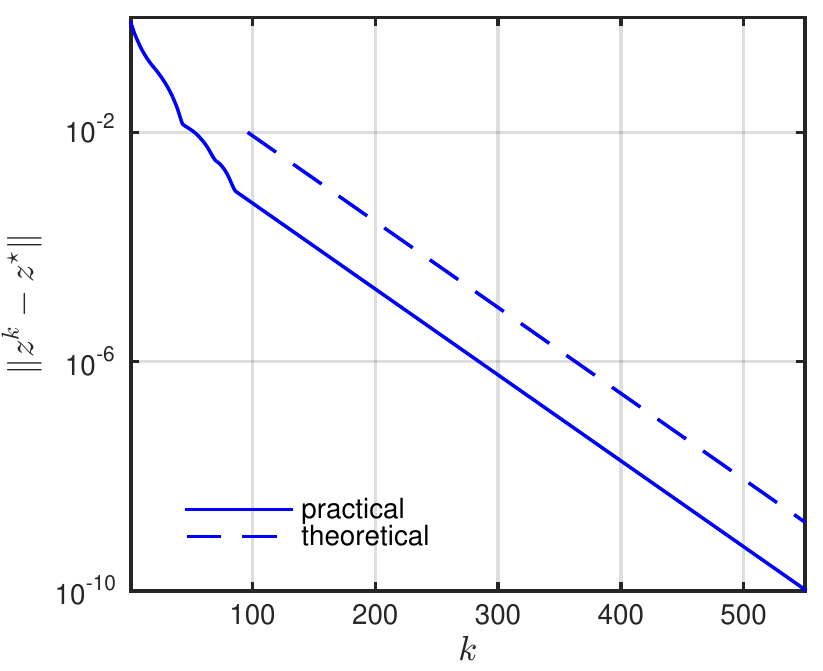}}
\subfigure[Outliers removal]{\includegraphics[width=0.32\textwidth]{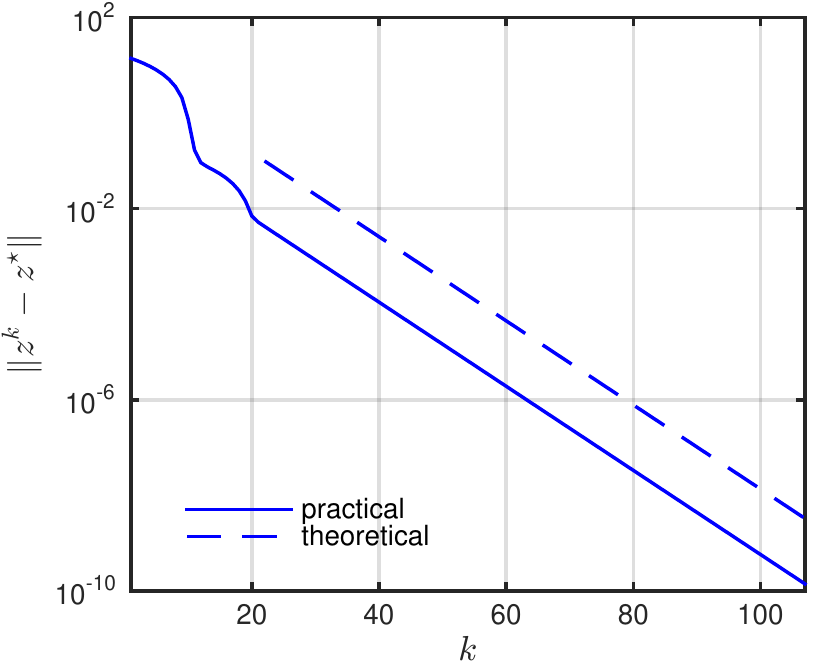}}
\caption{Observed (dashed) and predicted (solid) convergence profiles of DR \eqref{eq:drsgen} in terms of $\norm{\zk - \zfix}$. (a) CS with $\ell_1$. (b) CS with $\ell_{1,2}$. (c) CS with $\ell_\infty$. (d) TV image inpainting. (e) Uniform noise removal by solving \eqref{eq:mintvlinf}. (f) Outliers removal by solving \eqref{eq:mintvl1}. The starting point of the solid line is the iteration at which the manifolds are identified.}
\label{fig:DR1}
\end{figure}

\paragraph{TV based Image Inpainting}
In this image processing example, we observe $y=Ax_0$, where $A$ is a binary mask operator. We aim at inpainting the missing regions from the observations $y$. This can be achieved by solving \eqref{eq:minbp} with $J$ the 2D anisotropic TV.
The corresponding convergence profile is depicted in Figure~\ref{fig:DR1}(d).

\paragraph{Uniform Noise Removal}
For this problem, we assume that we observe $y = x_0 + \epsilon$, where $x_0$ is a piecewise-smooth vector, and $\epsilon$ is a realization of a random vector whose entries are iid $\sim \Uu([-a,a])$, $a > 0$. It is then natural to solve the problem
\begin{equation}\label{eq:mintvlinf}
\min_{x \in \RR^n} \norm{x}_{\mathrm{TV}} \qsubjq \norm{y - x}_{\infty} \leq a  .
\end{equation}
$G$ is now identified with the indicator function of the $\ell_\infty$-ball constraint, which is polyhedral and simple. The local convergence profile is shown in Figure~\ref{fig:DR1}(e) where we set $a=1$ and $n=100$. 

\paragraph{Outliers Removal}
Consider solving 
\begin{equation}\label{eq:mintvl1}
\min_{x \in \RR^n} \norm{y - x}_{1} + \lambda\norm{x}_{\mathrm{TV}} , 
\end{equation}
where $\lambda > 0$ is the tradeoff parameter. This problem has been proposed by~\cite{Nikolova04} for outliers removal. We take $J=\lambda\norm{\cdot}_{\mathrm{TV}}$ and $G=\norm{y - \cdot}_{1}$, which is again simple and polyhedral. For this example we have $n=100$, and $y-x$ is $10$-sparse, the corresponding local convergence profile is depicted in Figure~\ref{fig:DR1}(f).

%% file: tex/sec-conclusion.tex
\section{Conclusion}\label{sec:conclusion}
In this paper, we first showed that the DR splitting has the finite manifold identification under partial smoothness. When the involved manifolds are affine/linear, we proved local linear convergence of DR. When the involved functions are locally polyhedral, the optimal convergence rate is established. This is confirmed and illustrated by several numerical experiments. 